\documentclass[11pt]{amsart}
\usepackage{graphicx, amssymb, amsmath,amsthm}

\newcommand{\mb}{\mathbb}
\newcommand{\mc}{\mathcal}

\def \a{\alpha}   \def \d{\delta}
   \def \e{\epsilon}
\def \s{\sigma} \def \l{\lambda}  
   
\def \k{\kappa}  
\def \di{\mathrm{dist}} \def \spa{\mathrm{span}}
 
\def \L{\Lambda}

\newtheorem{theorem}{Theorem}[section]

\newtheorem{remark}{Remark}[section]

\newtheorem{lemma}[theorem]{Lemma}

\newtheorem{cor}[theorem]{Corollary}
\newtheorem{definition}[theorem]{Definition} 

\begin{document}

 \keywords{SRB measure, infinite dimensional dynamical systems, partial hyperbolicity, chaotic behavior.  \\  This work is partially supported by grants from CNSF and NSF}


 \author{Zeng Lian} \address[Zeng Lian] {School of Mathematical Sciences\\ Sichuan University\\
    Chengdu, 610065, China} \email[Z.~Lian]{zenglian@gmail.com}

\author{Peidong Liu} \address[Peidong Liu] {School of Mathematical Sciences\\ Peking University\\
    Beijing, 100871, P. R. China} \email[P.~Liu]{lpd@pku.edu.cn}

\author{Kening Lu} \address[Kening Lu] {  Department of Mathematics\\
 Brigham Young University\\
 Provo, Utah 84602, USA}
\email[k.~Lu]{klu@math.byu.edu}

\title[Existence of SRB Measures for  A Class of Partially
 Hyperbolic Attractors in Banach spaces]{Existence of SRB Measures for  A Class of Partially
 Hyperbolic Attractors in Banach spaces}

\pagestyle{plain}

\begin{abstract} {In this paper, we study the existence of SRB measures for infinite dimensional dynamical systems in a Banach space. We show that if the system has a partially hyperbolic attractor with nontrivial finite dimensional unstable directions, then it has an SRB measure.}
\end{abstract}
 \maketitle

\section{Introduction}

In smooth ergodic theory of finite dimensional dynamical systems, SRB measures (named after Sinai, Ruelle and Bowen who discovered them for uniformly hyperbolic attractors) are technically
defined as those invariant measures which have smooth conditional measures on  unstable manifolds, and this property  was in turn characterized as satisfying Pesin entropy formula (which, roughly
speaking, is an equality between entropy and  exponential volume expansion rate along unstable manifolds, see \cite{P}) (\cite{LY1}).  When the system is dissipative and there is no zero Lyapunov exponent,
SRB measures describe  the asymptotic behaviors of orbits with initial points in a positive Lebesgue measure set and thus are recognized as being physically significant. On the other hand, the theory
of SRB measures has led to significant new ideas in nonequilibrium statistical mechanics (\cite{Ruelle99}).

Which dynamical systems have SRB measures ? This has always been a challenging problem since their discovery in the 1970s.  Several classes of results have been obtained in this direction, including
partially hyperbolic attractors, H\'enon-like attractors, strange attractors arising from Hopf bifurcations etc.  Here we do not try to cover these progresses in detail but rather refer to the survey
paper \cite{Young2}, the relevant Palis conjecture \cite{Palis} and the book \cite{BDV}.

For better understanding of dynamical behaviors on attractors of dissipative partial differential equations, a program of extending the ideas of smooth ergodic theory of finite dimensional dynamical systems, especially of SRB measures,  to infinite dimensional setting was proposed by Eckmann and Ruelle \cite{ER}.  Several progresses have been made in this direction, among them we mention
\cite{M},  \cite{R}, \cite{T}, \cite{LL}, \cite{LSh}, \cite{LY}, \cite{LWY}, \cite{LLL} and \cite{AY}.  As for existence of SRB measures, we refer to \cite{LWY} and \cite{LLL}.

As a sequel to  \cite{LLL} which mainly deals with Hilbert spaces, this paper  is devoted to the existence of SRB measures for infinite dimensional systems in a separable Banach space.
Employing ideas of \cite{PS} and \cite{Young3},  we construct SRB measures on  a  partially hyperbolic attractor  of a differentiable map in such a Banach space
with nontrivial finite dimensional unstable directions.  Our aim is to understand dynamical behaviors of  the time-one map or a Poincar\'e section map of the solution flow for dissipative partial differential equations
(for example,  a parabolic one) with a Banach phase space such as $L^p$, $p\not=2$ (see,  for instant, \cite{Henry}).
As a consequence of our result,  a  partially hyperbolic attractor under consideration is chaotic in the sense that it contains a full weak horseshoe as introduced in \cite{HuangLu}, since
Pesin entropy formula proved by  \cite{AY} in case of Banach spaces together with the variational principle gives positive topological entropy which, by \cite{HuangLu}, implies existence of
such a horseshoe.
We remark that, as in the finite dimensional case, finding
concrete examples of PDEs or ODEs, to whose time-one maps or Poincar\'e section maps the results are applicable, is possibly a more challenging problem.

In \cite{LLL}, we established the existence of SRB measures and their basic properties for partially hyperbolic attractors in a separable Hilbert space, where the Lebesgue measures and Jacobians play key roles  as in all the previous work on SRB measures. In an Euclidean space or a Hilbert space, an inner product uniquely induces a system of Lebesgue measures in a natural way, and then the corresponding Jacobians are naturally defined and possess sufficiently nice properties; while in Banach spaces, there is no such obvious choice on the systems of Lebesgue measures. In this case, for each fixed system of Lebesgue measures, it is pointwise defined and only possesses certain regularity. This makes the method used in \cite{LLL} not  work, since one can not expect the density function of the target measure to be continuous any more and it is difficult to build the connection between the weak* limit properties of the pushed forward Lebesgue measures on unstable manifolds and the target SRB properties. To overcome this, we employ a much more delicate way which is inspired by Rohlin \cite{Rok}. We need to reformulate the concept of Lebesgue measures and Jacobians for finite dimensional objects (such as subspaces or manifolds) and maps between those objects respectively.  Some needed material about  Lebesgue measures in Banach spaces is
 given in Appendix A.  In section 2 we introduce the set-up and the main result. Section 3 is devoted to the proof of the result.

 \section{Settings and Main Results}  \label{S:Setting}
 Let $(\mb X,|\cdot|)$ be a separable Banach space, $f:\mb X \to \mb X$ be a $C^2$ map. Let $Df_x$ be the Fr\'{e}chet derivative of $f$ at point $x\in \mb X$.
The conditions below are assumed throughout:
 \begin{itemize}
 \item[C1)] $f$ is injective;
 \item[C2)] There exist an $f$-invariant compact set $\Lambda$, on which $Df_x$ is (i) injective,\\
  and (ii) for all $x\in \Lambda$ $$\kappa(x):=\limsup_{n\to\infty}\frac1n\log\|Df^n(x)\|_{\kappa}<0,$$ where $\|\cdot\|_{\k}$ is the Kuratowski measure of noncompactness of an operator;
 \item[C3)] $\Lambda$ is an attractor with basin $U$, i.e., $U$ is an open neighborhood of $\Lambda$ and $$\cap_{n\ge0}f^n(U)=\Lambda.$$
 \end{itemize}

 Recall that, for a linear operator $T$, $\|T\|_\k$ is defined to be the infimum of the set of numbers $r>0$ where $T(B)$, $B$ being the unit ball, can be covered by a finite number of balls of radius $r$. Since $\|T_2\circ T_1\|_\k\le\|T_2\|_\k\|T_1\|_\k$ and $\|Df\|_{\k}\le \|Df\|$ is uniformly bounded on $\L$, the limit in the definition of $\k(x)$ exists for any $x\in \L$ and is a measurable function. Also note that, by definition, $\|T\|_{\k}=-\infty$ if $T$ is compact.

 \begin{remark}\label{R:WeakerCondition}
 Note that some well-known results follow from the assumptions above immediately:
 \begin{itemize}
 \item[(i)] By the compactness of $\L$, there is always an $f$-invariant probability measure supported on $\L$,  which we denote by $\mu$.
  \item[(ii)] By applying the Multiplicative Ergodic Theorem, there is a full measure set $\L'\subset \L$, on which the notion of Lyapunov exponents can be introduced, we refer the reader to \cite{M}, \cite{T} and \cite{LL} for details.
\item[(iii)] Condition  C2) implies that, for all $x\in\L'$, there are at most finitely many non-negative Lyapunov exponents.
\item[(iv)] Attractors are important because they capture the asymptotic behavior of large
sets of orbits. In general, $\Lambda$ itself tends to be relatively small (compact and of finite
Hausdorff dimension) while its attraction basin, which by definition contains an open set, is
quite visible in the phase space. Notice that our attractors are not necessarily
global attractors in the sense of \cite{Hale} and \cite{Te}.
 \end{itemize}
 \end{remark}


 For given $x\in \Lambda$, we define the unstable set of $x$ as the following:
 \begin{align*}
  W^u(x)&=\{y\in \mb X | f^{-n}(y) \text{ exists }, \forall n\in \mb N,  |f^{-n}(y)-f^{-n}(x)|\to 0  \text{ exponentially fast as } n\to +\infty\}.
 \end{align*}
 By C3), we have $ W^u(x) \subset \Lambda$ for all $x\in\Lambda$.

 In general, 
 $W^u(x)$ is an immersed manifold rather than an embedded one. To avoid these disadvantages, one can study the local invariant manifolds defined below:
 \begin{align*}
 & W^{u}_{r(x)}(x)=\{y\in B(x,r(x))|\  f^{-n}(y) \text{ exists for all } n\in \mb N, |f^{-n}(y)-f^{-n}(x)|e^{-n\l(x)}\le C(x), n\ge 0\},
 \end{align*}
   where $B(x,r)$ is the $r$-ball centered at $x$, $r$ and $C$ are positive tempered functions and $\l$ is an $f$-invariant function. In particular, given an $f$-invariant measure $\mu$ with finite many positive Lyapunov exponents, there are measurable tempered functions $r,C: \L\to \mb R^+$ such that for $\mu$-a.e. $x\in \L$, $ W^{u}_{r(x)}(x)$ is an embedded finite dimensional disc with well controlled distortions.\\

   In the current setting, one quick observation is that, for $\mu$-a.e. $x$
   $$ W^u(x)=\bigcup_{n=0}^{+\infty}f^n_{f^{-n}(x)}( W^{u}_{r(f^{-n}(x))}(f^{-n}(x))),$$
    which implies that $ W^u(x)$ is an immersed manifold, since $ W^u_{r(\cdot)}(\cdot )$ is finitely dimensional and $f$ is injective and differentiable.

    By the definition of $ W^u(x)$, it is obvious that 
    $ W^u(x)\cap  W^u(y)\neq \emptyset$ if and only if $ W^u(x)= W^u(y)$. So, up to a $\mu$-null set, $\bigcup_{x\in \L} W^u(x)$ form a partition of $\Lambda$. Unfortunately, this partition may be not measurable. We need to introduce the following concepts:

 \begin{definition}\label{D:Subordinate}
 Let $\mu$ be an $f$-invariant  Borel probability measure on $\Lambda$. A measurable partition ${\mathcal P}$
of $\Lambda$ is said to be {\it subordinate to the unstable manifolds} with respect to $\mu$ if, for $\mu$-a.e. $x \in \Lambda$, one has that ${\mathcal P}(x) \subset W^u(x)$ (here ${\mathcal P}(x)$ denotes the element of ${\mathcal P}$ which contains $x$) and it contains an open neighborhood of $x$ in $ W^u(x)$ (endowed with the submanifold topology).
\end{definition}

\begin{definition}\label{D:SRB}
 An $f$-invariant  Borel  probability measure $\mu$ on $\Lambda$ is called an {\bf SRB measure} if for every measurable partition ${\mathcal P}$
of $\Lambda$  subordinate to the unstable manifolds with respect to $\mu$ one has
$$
\mu_x^{\mathcal P} \ll {\rm Leb}_x
$$
for $\mu$-a.e.  $x \in \Lambda$, where $\mu_x^{\mathcal P}$ denotes the conditional measure of $\mu$ on ${\mathcal P}(x)$ and Leb$_x$ denotes a Lebesgue measure on $W^u(x)$ induced by norm of $\mb X$.
\end{definition}

We call  $f|_{\Lambda}$  to be  {\bf partially hyperbolic} if the following holds:
for every $x \in \Lambda$ there is a splitting
\[
\mb X=E_x^u \oplus E_x^{cs}
\]
which depends continuously on $x \in \Lambda$  with $\dim E_x^u >0$ and satisfies that for every $x \in \Lambda$
\[
Df_{x} E_x^u=E_{fx}^u, \ \ \ Df_{x} E_x^{cs} \subset E_{fx}^{cs}
\]
and
\begin{equation}\label{E:PartialHyperbolic}
\left\{\begin{array}{ll}
|Df_{x} \xi| \geq e^{\lambda_0} |\xi|, \ \ \ &\forall\ \xi \in  E_x^u, \\
|Df_{x}\eta| \leq  |\eta|, \ \ \ &\forall\ \eta \in  E_x^{cs},
\end{array}
\right.
\end{equation}
where $\lambda_0>0$ is a constant.


The following is the main result we derived in this paper.

    \begin{theorem}\label{T:SRBExist}
    If $f|_\L$ is partially hyperbolic, then there exists at least one SRB measure of $f$ with support in $\L$.
    \end{theorem}

\section{Proof of Theorem \ref{T:SRBExist}.} \label{S:ExistSRB}

In this section, we assume that $f|_\L$ is partially hyperbolic and prove Theorem  \ref{T:SRBExist}.
\subsection{ Unstable manifolds for partially hyperbolic systems}\label{S:SUManifoldUPH}
In this section, we state  a version of local unstable manifolds theorem for partially hyperbolic systems.


\begin{lemma}\label{L:UnstableManifold}
  There exists a continuous family of $C^2$ embedded $k$-dimensional discs $\{W_\delta^u(x)\}_{x \in \Lambda}$ such that the following hold true for each $x \in \Lambda$:
\begin{itemize}
\item[(1)] $W_\delta^u(x)=\exp_x\left({\rm Graph} (h_x)\right)$ where
$$
h_x:  E_x^u (\delta)  \to E_x^{cs}
$$
is a $C^{2}$ map with $h_x(0)=0$, $Dh_x(0)=0$, $\|Dh_x\| \leq \frac{1}{3}$, $\|D^2h_x\|$ being uniformly bounded on $x$ and $E_x^u (\delta)=   \{ \xi \in E_x^u: |\xi|<\delta\}$;

\item[(2)] $f W_\delta^u(x) \supset W_\delta^u(f(x))$ and $W^u(x)=\bigcup_{n \geq 1} f^n W_\delta^u(x_{-n})$
where $x_{-n}$ is the unique point in $\Lambda$ such that $f^nx_{-n}=x$;

\item[(3)] $d^u(y_{-n}, z_{-n}) \leq \gamma_0 e^{-n(\lambda_0-\varepsilon_0)} d^u (y, z)$ for any $y,\ z \in  W_\delta^u(x)$, where $d^u$ denotes the distance along the unstable discs, $y_{-k}$ is the unique point in $\Lambda$ such that $f^ky_{-k}=y$, $z_{-k}$ is defined similarly and $\gamma_0>0$, $0<\varepsilon_0 <<\lambda_0$ are some constants;

\item[(4)] there is $0<\rho <\delta$ such that, if  $W_\rho^u(x):=Exp_x\left({\rm Graph} (h_x|_{E_x^u (\rho)})\right)$ intersects $W_\rho^u(\bar{x})$ for $\bar{x} \in \Lambda$, then
$W_\rho^u(x) \subset W_\delta^u(\bar{x})$, where $Exp_x$ is the affine map from the tangential fibre attached on $x$ to the phase space, which can be simply identified by the operation of adding $x$.
\end{itemize}
\end{lemma}
The existence of the local unstable manifold is following from the result of Section 9 in \cite{LL}; and to prove the other parts, since one do not need to use inner product in particular, so the same proof of Lemma 3.1 in \cite{LLL} works here, so we omit the proof of this lemma.

\medskip

\subsection{Proof of Theorem \ref{T:SRBExist}.}\label{SS:SRBExist}
 We construct an SRB measure $\mu$ by taking a weak* limit of the average of a pushed forward Lebesgue measure on a local unstable manifold. \\

 First, since Lebesgue measure defined for a normed space  depends  on the choice of bases, we set up a system of piecewisely continuous unit bases of $E^u$  and fix such a system in the rest argument.

 Since the splitting $\mb X=E_x^u\oplus E_x^{cs}$ varies continuously in $x$, for any small $\e\in(0,1)$ and $x\in \Lambda$, there exists $\d>0$ such that one can choose a unit basis $\eta_y=\{v_i(y)\}_{1\le i\le \dim E^u}$ of $E^u_y$ for any $y\in B(x,\d)\cap \Lambda$ with
\begin{equation}\label{E:Basis1}
\di(v_i(y),\spa\{v_j(y)\}_{1\le j\le i-1})>1-\e,\ 2\le i\le \dim E^u
\end{equation}
 and,  for each $1\le i\le \dim E^u$, $v_i(y)$ being  continuous in $y$. By Lemma 4.1 of \cite{LL}, we have that
 \begin{equation}\label{E:Basis2}
\di(v_i(y),\spa\{v_j(y)\}_{j\neq i})\ge\left(\frac{1-\e}{2-\e}\right)^{\dim E^u-1}(1-\e).
\end{equation}
Noting that $\Lambda$ is compact, there exist finite points $\{x_i\}_{i\in I}$ where $I$ is a finite index set such that $\Lambda\subset \cup_{i\in I} B(x_i,\d)$. These balls generate a finite partition by taking intersections. By choosing basis properly, we have constructed piecewisely continuous bases of the unstable linear fibres $\{E^u_x\}_{x\in \Lambda}$.

  Let $J^u(x)= |\det_{\eta_{x},\eta_{fx}} (Df_x|_{E^u_x})|$ for $x \in \Lambda$ , which is defined by (\ref{D:determinantApp}) in Section \ref{S:AppLeb}.


\begin{lemma}\label{L:DetUniform}
There is a constant $C>0$ such that for any $x\in\Lambda$, $y,\ z \in W_\delta^u(x)$ and $n \geq 1$
$$
\frac{1}{C} \leq \prod\limits_{k=1}^n
\frac{J^u(y_{-k})}{J^u(z_{-k})} \leq C,
$$
where $y_{-k}$ is the unique point in $\Lambda$ such that $f^ky_{-k}=y$ and $z_{-k}$ is defined similarly.
\end{lemma}

\begin{proof}
The proof is based on the Lipchitz continuity of $Df$ and of the subspaces $E^u$ restricted on unstable manifolds. Note that, since $\Lambda$ is compact, $f$ is $C^2$, and the splitting $E^u\oplus E^{cs}$ is continuous, $\|Df\|,\|\pi^u\|,\|\pi^{cs}\|$ are uniformly bounded. For sake of convenience,  we assume $0<\d\le 1$. By applying Lemma \ref{L:UnstableManifold}, we have that there exists $M\ge 1$ such that
$$\max\left\{\sup_{x\in\Lambda}\{\|Df_x\|\}, Lip(Df|_{\L}),\sup_{x\in\Lambda}\{\|\pi^u_x\|,\|\pi^{cs}_x\|\}, \sup_{x\in\Lambda}\{Lip\ Dh_x\} \right\}\le M.$$ For sake of convenience, we also assume that
$$\left(\frac{1-\e}{2-\e}\right)^{\dim E^u-1}(1-\e)>\frac1M.$$

First, for $k=0,1,\ldots$, we define linear operators $P_k^y:E^u_{x_{-k-1}}\to E^u_{y_{-k-1}}$ and $Q_k^y:E^u_{x_{-k-1}}\to E^u_{x_{-k}}$ as the following
\begin{align*}
&P_k^y=\left(I+Dh_{x_{-k}}\left(\pi^u_{x_{-k}}Exp_{x_{-k}}^{-1}(y_{-k})\right)\right)\big|_{E^u_{x_{-k}}},\\
&Q_k^y=\pi^u_{x_{-k}}(Df_{z_{-k-1}}|_{E^u_{y_{-k-1}}})P_{k+1}^y.
\end{align*}

 By applying (\ref{E:DetProdForm}), we have that
\begin{align*}
\det_{\eta_{x_{-n},\eta_{x_0}}}\left(\prod_{k=0}^{n-1}Q^y_k\right)
=&\prod_{k=0}^{n-1}\det_{\eta_{y_{-k}},\eta_{x_{-k}}}(\pi^u_{x_{-k}}|_{E^u_{y_{-k}}})J^u(y_{-k-1})
\det_{\eta_{x_{-k-1}},\eta_{y_{-k-1}}}(P_{k+1}^y)\\
=&\prod_{k=0}^{n-1}J^u(y_{-k-1})\prod_{k=1}^{n-1}\left(\det_{\eta_{x_{-k}},\eta_{y_{-k}}}(P_{k}^y)
\det_{\eta_{y_{-k}},\eta_{x_{-k}}}(\pi^u_{x_{-k}}|_{E^u_{y_{-k}}})\right)\\
&\times \det_{\eta_{x_{-n}},\eta_{y_{-n}}}(P_{n}^y)
\det_{\eta_{y},\eta_{x}}(\pi^u_{x}|_{E^u_{y}})\\
=&\det_{\eta_{x_{-n}},\eta_{y_{-n}}}(P_{n}^y)\det_{\eta_{y},\eta_{x}}(\pi^u_{x}|_{E^u_{y}})\prod_{k=0}^{n-1}J^u(y_{-k-1})
\prod_{k=1}^{n-1}\det_{\eta_{x_{-k}},\eta_{x_{-k}}}(\pi^u_{x_{-k}}|_{E^u_{y_{-k}}}P_{k}^y).
\end{align*}
By simple computation and (3) of Lemma \ref{L:UnstableManifold}, we obtain that
\begin{align}\begin{split}\label{E:ProjNormExtJu}
&\left\|\pi^u_{x_{-k}}|_{E^u_{y_{-k}}}P_{k}^y-I|_{E^u_{x_{-k}}}\right\|\le M^3|y_{-k}-x_{-k}|\le 4M^3\d\gamma_0 e^{-k(\lambda_0-\varepsilon_0)} ;\\
&\max\{\|P_n^y\|,\|(P_n^y)^{-1}\|,\|\pi^u_{x}|_{E^u_{y}}\|,\|(\pi^u_{x}|_{E^u_{y}})^{-1}\|\}\le \frac32.
\end{split}\end{align}
Also note that, by applying (\ref{E:determinantApp}), we have that
$$J^u(y_{-k})\ge (\dim E^u)^{-\frac12\dim E^u}M^{-\dim E^u}e^{\dim E^u\l_0}.$$
Now, by applying Lemma \ref{L:determinantApp}, there is a constant $C_1>1$ which  depends on $M$, $\dim E^u$ and $\l_0$ only such that
\begin{equation}\label{E:DetComparation1}
\frac1{C_1}\le \frac{\det_{\eta_{x_{-n},\eta_{x_0}}}\left(\prod_{k=0}^{n-1}Q^y_k\right)}{\prod_{k=0}^{n-1}J^u(y_{-k-1})}\le C_1.
\end{equation}
Next, we define $P^z_k$ and $Q^z_k$ analogously and compare $Q^z_k$ with $Q^y_k$. By simple computation, we have
\begin{equation}\label{E:ProjNormExtJu1}
\|Q^z_k-Q^y_k\|\le \frac73 M^2|z_{-k}-y_{-k}|.
\end{equation}
Note that
$$\min\left\{\det_{\eta_{x_{-k-1}},\eta_{x_{-k}}}(Q_k^z),\det_{\eta_{x_{-k-1}},\eta_{x_{-k}}}(Q_k^y)\right\}\ge \frac1{C_1}(\dim E^u)^{-\frac12\dim E^u}M^{-\dim E^u}e^{\dim E^u\l_0}.$$
Then, by applying (\ref{E:DetLipApp}) in Lemma \ref{L:determinantApp}, we have
\begin{equation}\label{E:DetComparation2}
\left|\frac{\det_{\eta_{x_{-k-1}},\eta_{x_{-k}}}(Q_k^z)}{\det_{\eta_{x_{-k-1}},\eta_{x_{-k}}}(Q_k^y)}-1\right|\le C_2|z_{-k}-y_{-k}|,
\end{equation}
where $C_2\ge 1$  depends only on $\dim E^u$ and $M$. Then there is a constant $C_3\ge 1$  depending only on $\dim E^u$, $M$ and $\l_0$ such that
\begin{equation}\label{E:DetComparation2}
\frac1{C_3}\le\frac{\det_{\eta_{x_{-n}},\eta_{x}}(\prod_{k=0}^nQ_k^z)}{\det_{\eta_{x_{-n}},\eta_{x}}(\prod_{k=0}^nQ_k^y)}\le C_3,
\end{equation}
which, together with (\ref{E:DetComparation1}) and (\ref{E:DetComparation2}), completes the proof.
\end{proof}

Fix a point $\hat{x} \in \Lambda$ and write $L=W_\delta^u(\hat{x})$. Let $\lambda_L$ be a normalized Lebesgue measure on $L$. Let $\mu$ be a limit measure of $\frac{1}{n} \sum_{k=0}^{n-1} f^k \lambda_L,n \geq 1$, and assume that
$$
\frac{1}{n_i} \sum_{k=0}^{n_i-1} f^k \lambda_L \to \mu
$$
as $i \to +\infty$ for some subsequence $\{n_i\}_{i \geq 1}$ of the positive integers. Note that the existence of $\mu$ follows from the compactness of $\Lambda$. We will show that such $\mu$ is an SRB measure.

Before starting the main proof,  for sake of convenience, we first assume
$$\lambda_L=\mu_{\eta_{\hat x}}\circ \pi^u_{\hat x}\circ Exp_{\hat x}^{-1},$$
where $\mu_{\eta_{\hat x}}$ is the Lebesgue measure on $E^u_{\hat x}$ induced by unit basis $\eta_{\hat x}$.
 By (1) of Lemma \ref{L:UnstableManifold}, it is easy to see that $\l_L$ is a well defined Lebesgue measure on $L$.

Let $x \in \Lambda$. Set $\Sigma_{x, \varepsilon}=Exp_x(E^{cs}_x (\varepsilon)) \bigcap \Lambda$ and let
$$
V_{x, \varepsilon}=\bigcup\limits_{y \in \Sigma_{x, \varepsilon}} W_\rho^u(y).
$$
By (4) of Lemma \ref{L:UnstableManifold},  we know that, when $\varepsilon$ is small enough, $ V_{x, \varepsilon}$ is a union of  pairwisely disjoint pieces of $W_\rho^u(y)$ with $y \in \Sigma_{x, \varepsilon}$  and it contains a neighborhood of $x$ in $\Lambda$. By Lemma \ref{L:UnstableManifold} and the continuity of the splitting $\mb X=E^{u}\oplus E^{cs}$, for small enough $\e>0$, each $W^{u}_{\rho}(y)\subset V_{x,\e}$ can be viewed as the graph of a $C^{2}$ function $h_{y}':E^{u}_{x}(\rho_{y})\to E^{cs}_{x}$ with uniform bounds of $\|Dh'_{\cdot}\|$ and  $\|D^{2}h'_{\cdot}\|$,  with $h_{y}'$ and $\rho_{y}$ varying continuously in $y$. By tailoring $V_{x,\e}$ a little bit, one can obtain a subset $V'_{x,\e}$ of $V_{x,\e}$ such that
$$V'_{x,\e}=\bigcup_{y\in \Sigma_{x,\e}} {\rm graph} (h'_{y}|_{E^{u}_{x}(\rho_{0})})$$
for some $\rho_{0}>0$ with $\rho_{0}\le \rho_{y}$ for all $y\in \Sigma_{x,\e}$. It is obvious that $ V'_{x, \varepsilon}$ is also a  union of pairwisely disjoint pieces of $W_\rho^u(y)$ with $y \in \Sigma_{x, \varepsilon}$  and it contains a neighborhood of $x$ in $\Lambda$. For sake of simplicity, we denote $W^{u}_{x,\rho_{0}}(y)={\rm graph}(h'_{y}|_{E^{u}_{x}(\rho_{0})})$ for each $y\in \Sigma_{x,\e}$.

  Since  $\Lambda$ is compact, we have a finite number of sets of this kind $\{ V'_{x,\e} \}$  which cover $\Lambda$. With a bit abuse of notation, let $V=V'_{x, \varepsilon}$ be an arbitrary one of these sets which satisfies $\mu(V)>0$. Moreover, since
$W^{u}_{x,\rho_{0}}(y)$ is contained in $\Lambda$ for every $y \in \Lambda$, by shrinking $\varepsilon$ and $\rho_{0}$ if necessary, we may assume that $\mu(\partial V)=0$ where $\partial V$ is the boundary of $V$ as a subset of $\Lambda$.
Note that  $\{W^{u}_{x,\rho_{0}}(y)\}_{y \in \Sigma_{x, \varepsilon}}$  produces a measurable partition of $V$, which is denoted by $\zeta$. Actually there exist countably many partitions $\{\zeta_{n}\}_{n\ge 1}$ of $V$ satisfying that
\begin{itemize}
\item[Par1)] For each $n\in \mb N$, $\zeta_{n}$ consists of finitely many elements;
\item[Par2)] For any $A\in \zeta_{n}$, $A=\cup_{y\in S}W^{u}_{x,\rho_{0}}$ for some $S\subset \Sigma_{x,\e}$, and $\mu(\partial A)=0$, where $\partial A$ is the boundary of $A$ as subset in $\Lambda$;
\item[Par3)] $\max_{A\in \zeta_{n}}\left\{\text{ diameter of } A\cap \Sigma_{x,\e}\right\}\to 0$ as $n\to \infty$.
\end{itemize}
It is easy to see that, up to a $\mu$-null set
\begin{equation}\label{E:Par4}
\zeta=\bigvee_{n=1}^{\infty} \zeta_{n}.
\end{equation}

 Let $(\mu|_V)_y$ be the conditional probability measure of $\mu|_V$ (the restriction of $\mu$ to $V$) on $W^{u}_{x,\rho_{0}}(y)$, and $\nu$ be the induced measure of $\mu|_{V}$ on quotient space $V/\sim$ where $\sim$ is an equivalent relation  that $z_{1}\sim z_{2}$ if and only if there exists $y\in \Sigma_{x,\e}$ such that $z_{1},z_{2}\in W^{u}_{x,\rho_{0}}(y)$. Note that $\nu$ also induces a measure on $\Sigma_{x,\e}$, and,  with a little abuse of notation, we will not distinguish them and use $\nu$ to denote both.

 It is then easy to see that
$\mu$ will be an SRB measure if,
neglecting a set of $\mu|_V$-null set, we have
\begin{equation}\label{E:(1.0)}
(\mu|_V)_y \ll \lambda^u_y
\end{equation}
on every piece $W^{u}_{x,\rho_{0}}(y)$,  where $\lambda^u_y$ is a Lebesgue measure on  $W^{u}_{x,\rho_{0}}(y)$. Again, for sake of convenience and without losing any generality, we let
$$\l^u_y=\mu_{\eta_{x}}\circ \pi^u_{x}\circ Exp_{x}^{-1},$$
where $\mu_{\eta_{x}}$ is the Lebesgue measure on $E^u_{x}$ induced by the unit basis $\eta_{x}$.

For each $n \geq 0$, let
$$L_n=\{ z \in L: f^nz \in  W^{u}_{x,\rho_{0}}(y)\ {\rm for\ some}\ y \in \Sigma_{x, \varepsilon}\
{\rm but}\ f^nL \not\supset  W^{u}_{x,\rho_{0}}(y)\}.$$
From (3) and (4) of Lemma \ref{L:UnstableManifold}, we have that, for any $z\in L_n$, $d^u(z,\partial L)\le \frac43\d\gamma_0e^{-n(\l_0-\e_0)}$. This is because, otherwise,  for any $z'\in W^{u}_{x,\rho_{0}}(y)$, by (4) of Lemma \ref{L:UnstableManifold}, $z'\in W^u_y(\d)$, then by (3) of Lemma \ref{L:UnstableManifold},  $$d^u(z,z'_{-n})\le\gamma_0e^{-n(\l_0-\e_0)}d^u(f^nz,z')\le \frac43\d\gamma_0e^{-n(\l_0-\e_0)},$$
which implies that $z'\in f^nL$ and thus $W^{u}_{x,\rho_{0}}(y)\subset f^nL$, resulting in a contradiction. Therefore, we know that $\lambda_L(L_n) \to 0$ exponentially fast as $n \to +\infty$. Thus
\begin{equation}\label{E:(1.1)}
\lim\limits_{i \to +\infty} \frac{1}{n_i} \sum_{k=0}^{n_i-1} f^k (\lambda_L|_{(L\setminus L_k)}) = \mu  \end{equation}
which together with the fact  $\mu(\partial V)=0$ implies
\begin{equation}\label{E:(1.2)}
\lim\limits_{i \to +\infty} \left( \frac{1}{n_i} \sum_{k=0}^{n_i-1} f^k (\lambda_L|_{(L\setminus L_k)})\right)(V) = \mu (V) .
\end{equation}

Suppose that $f^n(L \setminus L_n) \supset W^{u}_{x,\rho_{0}}(y)$  for some  $y \in \Sigma_{x, \varepsilon}$. Let $m_{n, y}$ be the conditional probability measure of  $ [f^n (\lambda_L|_{(L\setminus L_n)})]|_V$ on $W^{u}_{x,\rho_{0}}(y)$. For $z\in V$, we  define
$$
p_{n}(z)=\begin{cases}\frac{d m_{n, y}}{d\lambda^u_y}(z)&\text{ if }z\in W^{u}_{x,\rho_{0}}(y)\subset f^n(L \setminus L_n) \\
1&\text{ otherwise}\end{cases}.
$$
By simple computation, we have that for $z\in W^{u}_{x,\rho_{0}}(y)\subset f^n(L \setminus L_n) $

\begin{align}\begin{split}\label{E:(1.3)}
p_{n}(z)&=\left|\frac{\frac{\det_{\eta_{z_{-n}},\eta_{\hat x}}\left(\pi^u_{\hat x}|_{E^u_{z_{-n}}}\right)}{\det_{\eta_x,\eta_z}\left(\pi^u_x|_{E^u_z}\right)}\prod\limits_{k=1}^n \frac{1}{J^u(z_{-k})}}{\int_{W^{u}_{x,\rho_{0}}(y)}\frac{\det_{\eta_{w_{-n}},\eta_{\hat x}}\left(\pi^u_{\hat x}|_{E^u_{w_{-n}}}\right)}{\det_{\eta_x,\eta_w}\left(\pi^u_x|_{E^u_w}\right)} \prod\limits_{k=1}^n \frac{1}{J^u(w_{-k})}\ d\lambda^u_y (w)}\right|\\
&=\left|\frac{\frac{\det_{\eta_{z_{-n}},\eta_{\hat x}}\left(\pi^u_{\hat x}|_{E^u_{z_{-n}}}\right)}{\det_{\eta_x,\eta_z}\left(\pi^u_x|_{E^u_z}\right)}\prod\limits_{k=1}^n \frac{J^u(y_{-k})}{J^u(z_{-k})}}{\int_{W^{u}_{x,\rho_{0}}(y)}\frac{\det_{\eta_{w_{-n}},\eta_{\hat x}}\left(\pi^u_{\hat x}|_{E^u_{w_{-n}}}\right)}{\det_{\eta_x,\eta_w}\left(\pi^u_x|_{E^u_w}\right)} \prod\limits_{k=1}^n \frac{J^u(y_{-k})}{J^u(w_{-k})}\ d\lambda^u_y (w)} \right|.
\end{split}\end{align}

For each $n \geq 0$,  $p_{n}: V \to (0, +\infty)$ is clearly measurable.
Noting that (\ref{E:ProjNormExtJu}) holds,  by applying the same argument as used in the proof of Lemma \ref{L:DetUniform}, we have that
$$\frac1C\le \left|\frac{\det_{\eta_{z_{-n}},\eta_{\hat x}}\left(\pi^u_{\hat x}|_{E^u_{z_{-n}}}\right)}{\det_{\eta_x,\eta_z}\left(\pi^u_x|_{E^u_z}\right)}\right|\le C,$$
where $C\ge 1$ is a constant which  depends only on the system constants. Therefore, by Lemma \ref{L:DetUniform} and the above estimate, we have that there exists a constant $C\ge 1$ which depends on the system constants only such that
\begin{equation}\label{E:HUniformBD}
\frac1C\le p_n\le C \text{ for all }n\ge 0.
\end{equation}

Now we prove (\ref{E:(1.0)}). First, we consider the cylindrical sets
$$A_{S,F}=\bigcup_{y\in S}Exp_{x}({\rm graph}(h'_{y}|_{F})),$$
where $S\subset \Sigma_{x,\e}$ and $F\subset E^{u}_{x}(\rho_{0})$ are  Borel subsets. Since $\mu_{\eta_x}$ is regular and $\mu_{\eta_x}(F)<\infty$, for any given small $\e'>0$ there exist a compact set $K$ and an open set $U$ such that
$$K\subset F\subset U\subset E^{u}_{x}(\rho_{0})\text{ and }\mu_{\eta_x}(U)-\mu_{\eta_x}(K)<\e'.$$
Consider a Borel set $S\subset \Sigma_{x,\e}$ satisfying
\begin{equation}\label{E:SZeroBD}
\mu\left(\partial\left(\cup_{y\in S} W^{u}_{x,\rho_{0}}(y)\right)\right)=0.
\end{equation}
We denote $\partial S$ the boundary of $S$ , $S^{o}$ the interior of $S$  and $\overline S$ the closure of $S$ as subset of $\overline\Sigma_{x,\e}$. By Lemma \ref{L:UnstableManifold} and also noting that $\mu(\partial V)=0$, it is easy to see that
$$\mu\left(\partial\left(\cup_{y\in S} W^{u}_{x,\rho_{0}}(y)\right)\right)=\mu(\cup_{y\in \partial S} W^{u}_{x,\rho_{0}}(y)).$$
Denote $\L_{k}=f^{k}(L\setminus L_{k})\cap \Sigma_{x,\e}$. Then,   by combining (\ref{E:Weak*1}), (\ref{E:(1.1)}), (\ref{E:(1.2)}), (\ref{E:HUniformBD}) and (\ref{E:SZeroBD}), and by applying Lemma \ref{L:ConvergenceInMeasure}, we obtain that
\begin{align*}
\mu(A_{S^{o},U})\le &\liminf_{i \to \infty} \left( \frac{1}{n_i} \sum_{k=0}^{n_i-1} f^k (\lambda_L|_{(L\setminus L_k)})\right)\Big|_V(A_{S^{o},U})\\
=&\liminf_{i\to\infty}\frac1{n_{i}}\sum_{k=0}^{n_i-1}\sum_{y\in\Lambda_k}\left(\left(f^k(\l_L|_{L\setminus L_k})(W^u_{x,\rho_{0}}(y)\cap A_{S^{o},U})\right)\int_{A_{S^{o},U}\cap W^u_{x,\rho_{0}}(y)}p_{k}(z)d\l^u_y(z)\right)\\
\le&\liminf_{i\to\infty}\frac1{n_{i}}\sum_{k=0}^{n_i-1}\sum_{y\in\Lambda_k}\left(\left(f^k(\l_L|_{L\setminus L_k})(W^u_{x,\rho_{0}}(y)\cap A_{S^{o},U})\right)\int_{A_{S^{o},U}\cap W^u_{x,\rho_{0}}(y)}Cd\l^u_y(z)\right)\\
=&C\mu_{\eta_x}(U)\liminf_{i\to\infty}\left(\frac1{n_{i}}\sum_{k=0}^{n_i-1}f^k(\l_L|_{L\setminus L_k})\right)\left(\bigcup_{y\in S^{o}}W^{u}_{x,\rho_{0}}(y)\right)\\
=&C\mu_{\eta_x}(U)\mu\left(\bigcup_{y\in S^{o}}W^{u}_{x,\rho_{0}}(y)\right)=C\mu_{\eta_x}(U)\mu\left(\bigcup_{y\in S}W^{u}_{x,\rho_{0}}(y)\right)\\
=&C\mu_{\eta_x}(U)\nu(S).
\end{align*}
By employing (\ref{E:Weak*2}) and applying the similar argument as above, one can obtain that
$$\mu(A_{\overline S,K})\ge \frac1{C}\mu_{\eta_x}(K)\mu\left(\bigcup_{y\in \overline S}W^{u}_{x,\rho_{0}}(y)\right)=\frac1{C}\mu_{\eta_x}(K)\nu(S).$$
By (\ref{E:SZeroBD}), we have that
\begin{align*}
\mu(A_{S,F})&\ge \mu(A_{\overline S,F})-\mu(\partial(\cup_{y\in S}W^{u}_{x,\rho_{0}}(y)))\ge \mu(A_{\overline S,K})\\
\mu(A_{S,F})&\le \mu(A_{ S^{o},F})+\mu(\partial(\cup_{y\in  S}W^{u}_{x,\rho_{0}}(y)))\le \mu(A_{ S^{o},U}).
\end{align*}
Noting that $\mu_{\eta_x}(U)-\mu_{\eta_x}(K)<\e'$ and $\e'$ can be taken arbitrarily small, hence
\begin{equation}\label{E:AbsCylinSet}
\frac1{C}\mu_{\eta_x}(F)\nu(S)\le \mu(A_{S,F})\le C\mu_{\eta_x}(F)\nu(S),
\end{equation}
for all cylindrical sets  $A_{S,F}\subset V$ with $S$ satisfying (\ref{E:SZeroBD}).

Now we are ready to show that for $\nu$-a.e. $y\in \Sigma_{x,\e}$, $(\mu|_{V})_{y}\ll\l_{y}^{u}$. Let $\xi=\{F_{n}\}_{n\in\mb N}$ be a collection of countably many open subsets of $E^{u}_{x}(\rho_{0})$ such that the minimal $\sigma$-algebra containing $\xi$ is the Borel $\sigma$-algebra. Such $\xi$ can be constructed by collecting  intersections of $E^{u}_{x}(\rho_{0})$ with balls centered at points from a  countable dense subset of $E^{u}_{x}(\rho_{0})$ with rational radii. Let $\xi_{d}$ be the minimal algebra containing $\xi$, it is easy to see that $\xi_{d}$ consists of countably many elements.
Denote $\zeta_{d}^0$ the minimal algebra containing $\bigcup_{n=1}^{\infty} \zeta_{n}^0$ where $\zeta_n^0=\{\text{interior of }A|\ A\in \zeta_n\}$ and $\zeta_n$ is as in (\ref{E:Par4}).
Note that $\zeta_{d}^0$ consists of countably many elements and induces an algebra on $\Sigma_{x,\e}\setminus\cup_{n=1}^\infty\partial\zeta_n$ by projection along unstable fibers,  which we denote by $\zeta^0_{\Sigma,d}$. The Borel $\s$-algebra on $\Sigma_{x,\e}\setminus\cup_{n=1}^\infty\partial\zeta_n$ can be generated by $\zeta^0_{\Sigma,d}$ because of properties Par3). Then, by the Carath\'eodory Extension Theorem and by employing the concept of outer measure and  property Par2),  we have that for any Borel measurable set $A\subset \Sigma_{x,\e}$
\begin{equation}\label{E:OuterMeasure}
\nu(A)=\inf\{\nu(S)|\ A\setminus\cup_{n=1}^\infty\partial\zeta_n\subset S\in \zeta^0_{\Sigma,d}\}.
\end{equation}
Obviously, each $S\in \zeta^0_{\Sigma,d}$ satisfies (\ref{E:SZeroBD}) because of property Par2). Given $F\in \xi_{d}$, by (\ref{E:AbsCylinSet}), we have that  for any $S\in \zeta^0_{\Sigma,d}$
$$\frac1{C}\mu_{\eta_x}(F)\nu(S)\le\mu(A_{S,F})=\int_{S}(\mu|_V)_{y}(A_{S,F}\cap W^{u}_{x,\rho_{0}}(y))d\nu(y)\le C\mu_{\eta_x}(F)\nu(S).$$
Then (\ref{E:OuterMeasure}) implies that
$$\frac1{C}\mu_{\eta_x}(F)\le(\mu|_V)_{y}(A_{\Sigma_{x,\e},F}\cap W^{u}_{x,\rho_{0}}(y))\le C\mu_{\eta_x}(F)\text{ for }\nu-\text{a.e.}\ y\in \Sigma_{x,\e}.$$
Since $\xi_{d}$ consists of countably many elements, there exists $\Sigma'\subset \Sigma_{x,\e}$ with $\nu(\Sigma')=\nu(\Sigma_{x,\e})$ such that for any $y\in \Sigma'$,
\begin{equation}\label{E:AbsLeb}
\frac1{C}\mu_{\eta_x}(F)\le(\mu|_V)_{y}(A_{\Sigma_{x,\e},F}\cap W^{u}_{x,\rho_{0}}(y))\le C\mu_{\eta_x}(F)\text{ for each } F\in \xi_{d}.
\end{equation}
Then,  by applying Carath\'eodory Extension Theorem and employing the concept of outer measure again, we have that for $\nu$-a.e. $y\in \Sigma_{x,\e}$ and any Borel measurable set $G\subset W^{u}_{x,\rho_{0}}$
\begin{equation}\label{E:OuterMeasure1}
(\mu|_V)_{y}(G)=\inf\left\{(\mu|_V)_{y}(A_{\Sigma_{x,\e},F}\cap W^{u}_{x,\rho_{0}}(y))|\ \pi_{x}^{u}Exp^{-1}(G)\subset F\in \xi_{d}\right\}.
\end{equation}
Also note that $\mu_{\eta_x}(\pi^{u}_{x}Exp^{-1}(G))=\inf\{\mu_{\eta_x}(F)|\ \pi^{u}_{x}Exp^{-1}(G)\subset F\in \xi_{d}\}$.
Therefore, from (\ref{E:AbsLeb}) and (\ref{E:OuterMeasure1}), we have  that for $\nu$-a.e. $y\in \Sigma_{x,\e}$ and any Borel measurable set  $G\subset W^{u}_{x,\rho_{0}}$
$$\frac1{C}\mu_{\eta_x}\left(\pi^{u}_{x}Exp^{-1}(G)\right)\le (\mu|_V)_{y}(G)\le C\mu_{\eta_x}\left(\pi^{u}_{x}Exp^{-1}(G)\right),$$
which implies that (\ref{E:(1.0)}) is true for $\mu|_V$-a.e. $y\in V$. Since the Lebsesgue measures defined on an embedded finitely dimensional disc are equivalent, the validity of (\ref{E:(1.0)}) does not depend on the choice of  Lebesgue measure $\l_y^u$. The proof is completed.

\appendix

\section{ Lebesgue Measures in Normed Spaces}\label{S:AppLeb}
It is well known that the Lebesgue measure is not intrinsically defined for finite dimensional normed spaces. For our purpose in this paper, we need to assign Lebesgue measures for finite dimensional subspaces in Banach spaces (which are normed spaces), and for $C^1$ graphs of maps whose domain is a  finite dimensional subspace, with which the determinant of  linear transformations between these subspaces are well defined. The following well-known Theorem is due to Haar:
\begin{theorem}\label{T:Haar}
If $\lambda$ is a translation-invariant measure on $\mb R^n$ for which all compact sets have finite measure and all open sets have positive measure, then $\lambda$ is a constant multiple of the Lebesgue measure.
\end{theorem}
The Radon-Nikodym derivative of one Lebesgue measure with respect to another defined for a given finite dimensional normed space is then a positive constant. The following results are not new, but for sake of completeness, we include the proof here. For a systemic introduction, we refer to the survey paper \cite{PT}.\\

Let $V,W$ be two $k$-dimensional normed spaces and $\mb R^k$ be the $k$-dimensional Euclidian space, $\mu$ be the Lebesgue measure on $\mb R^k$ such that $\mu([0,1]^k)=1$. Let $\eta_V=\{v_i\}_{1\le i\le k}\subset V,\eta_W=\{w_i\}_{1\le i\le k}\subset W$ be unit bases. Then we define
$$L_{\eta_V}:\mb R^n\to V,\ L_{\eta_V}e_i=v_i,\text{ and } L_{\eta_W}:\mb R^n\to W,\ L_{\eta_W}e_i=w_i.$$
Since $L_V,L_W$ are linear homeomorphisms, they induce complete measures $\mu_{\eta_V},\mu_{\eta_W}$ and $\sigma$-algebras of Lebesgue measurable sets on $V,W$ respectively, for which all compact sets have finite measure and all open sets have positive measure. Let $T:V\to W$ be a linear operator (thus bounded since the spaces have finite dimension). With $\mu_{\eta_V}$ and $\mu_{\eta_W}$, and also noting that $L_{\eta_W}^{-1}TL_{\eta_V}$ is an $n\times n$ matrix, we define
\begin{equation}\label{D:determinantApp}
\det_{\eta_V,\eta_W} (T)=\det (L_{\eta_W}^{-1}TL_{\eta_V}).
\end{equation}
A immediate consequence of this definition is the following property: let $T_1:V_1\to V_2$ and $T_2:V_2\to V_3$ be linear operators between $k$-dimensional normed linear spaces and $\eta_1,\eta_2,\eta_3$ be unit basis of $V_1,V_2,V_3$ respectively, then
\begin{equation}\label{E:DetProdForm}
\det_{\eta_1,\eta_3}(T_2 T_1)=\det_{\eta_1,\eta_2}(T_1)\det_{\eta_2,\eta_3}(T_2).
\end{equation}
The following lemma gives the relation between the norm and determinant of operators.
\begin{lemma}\label{L:determinantApp}
Suppose $\min\{\di(v_i,\spa\{v_j\}_{j\neq i}),\di(w_i,\spa\{w_j\}_{j\neq i})\}\ge\a>0$, then
\begin{equation}\label{E:determinantApp}
\left|\det_{\eta_V,\eta_W} (T)\right|=\left|\det (L_{\eta_W}^{-1}TL_{\eta_V})\right|\le k^{\frac k2}\|T\|^k\a^{-k},
\end{equation}
and $\det_{\eta_V,\eta_W} (\cdot): L(V,W)\to \mb R$ is locally Lipchitz, moreover, for any $T_1,T_2\in L(V,W)$,
\begin{equation}\label{E:DetLipApp}
\left|\det_{\eta_V,\eta_W} (T_2)-\det_{\eta_V,\eta_W} (T_1)\right|\le k^{\frac k2+1}(\max\{\|T_1\|,\|T_2\|\})^{k-1}\a^{-k}\|T_2-T_1\|.
\end{equation}
\end{lemma}
\begin{proof}
Denote $a_{ij}=ent_{ij}L_{\eta_W}^{-1}TL_{\eta_V}$. Then for any $1\le i\le k$, it is easy to see that
$$Tv_i=\sum_{i=1}^ka_{ij}w_j\text{ with } |a_{ij}|\le \|T\|\a^{-1}.$$
Then one has
$$
\left|\det (L_{\eta_W}^{-1}TL_{\eta_V})\right|\le\prod_{j=1}^k\left(\sum_{i=1}^ka_{ij}^2\right)^{\frac12}\le k^{\frac k2}\|T\|^k\a^{-k}.
$$
For (\ref{E:DetLipApp}), denoting $a_{\tau,ij}=ent_{ij}L_{\eta_W}^{-1}T_\tau L_{\eta_V},\tau=1,2$, then
$$|a_{2,ij}-a_{1,ij}|\le \|T_2-T_1\|\a^{-1}, |a_{\tau,ij}|\le \|T_\tau\|\a^{-1}.$$
It is straightforward to derive that
\begin{align*}
&\left|\det_{\eta_V,\eta_W} (T_2)-\det_{\eta_V,\eta_W} (T_1)\right|\\
\le&\sum_{j=1}^k\left(\prod_{p=1}^{j-1}\left(\sum_{i=1}^ka_{2,ip}^2\right)^{\frac12}
\prod_{q=j+1}^k\left(\sum_{i=1}^ka_{1,iq}^2\right)^{\frac12}
\left(\sum_{i=1}^k(a_{2,ij}-a_{1,ij})^2\right)^{\frac12}\right)\\
\le& k^{\frac k2+1}(\max\{\|T_1\|,\|T_2\|\})^{k-1}\a^{-k}\|T_2-T_1\|.
\end{align*}
\end{proof}

 Obviously, these measures depend on the choice of  the bases. However, the ratio of two such measures can be controlled.  We  summarize it as the corollary below which follows from Lemma \ref{L:determinantApp} if one take $V=W$ and $T=id$.  To save notations,  we still use $\eta_V,\ \eta_W$ as two arbitrary unit bases and $\mu_{\eta_V},\  \mu_{\eta_W}$ the corresponding induced measures.
\begin{cor}\label{C:RatioApp}
Suppose $\min\{\di(v_i,\spa\{v_j\}_{j\neq i}),\di(w_i,\spa\{w_j\}_{j\neq i})\}\ge\a>0$, then
there exists a constant $K>0$ such that for any Lebesgue measurable set $A\subset V$ with $\mu_{\eta_V}>0$,
$$\frac{\mu_{\eta_W}(A)}{\mu_{\eta_V}(A)}=K\le k^{\frac k2}\a^{-k}.$$
\end{cor}
In the case of $V=W$, let $T:V\to V$ be such that $T(v_i)=u_i,\ 1\le i\le k$ where $\eta_V=\{v_i\}_{1\le i\le k}$ is a unit basis of $V$. Then we have the following lemma on continuous dependence of the determinant on the choice of basis.
\begin{lemma}\label{L:ContMeaApp}
Suppose $\di(v_i,\spa\{v_j\}_{j\neq i})\ge\a>0$, then
\begin{equation}\label{E:ContOnBase1App}
\left|\det_{\eta_V,\eta_V} (T)-1\right|\le k^{\frac {3k}2+3}\a^{-k-2}\sup_{1\le i\le k}|u_i-v_i|.
\end{equation}
Furthermore, if $\eta_{W}=\{u_i\}_{1\le i\le k}$ forms a unit basis, then for any Lebesgue measurable set $A\subset V$,
\begin{equation}\label{E:ContOnBase2App}
\mu_{\eta_V}(A)= |\det_{\eta_V,\eta_V} (T)|\mu_{\eta_W}(A)\le \left(1+k^{\frac {3k}2+3}\a^{-k-2}\sup_{1\le i\le k}\{|u_i-v_i|\}\right)\mu_{\eta_W}(A).
\end{equation}
\end{lemma}
\begin{proof}
This Lemma follows from Lemma \ref{L:determinantApp}. To derive (\ref{E:ContOnBase1App}), one needs to estimate $\|T\|$ and $\|T-id\|$, which can be done as follows: For any $v\in V$ with $|v|=1$, there exists $\{a_i\in \mb R|1\le i\le k\}$ such that $v=\sum_{i=1}^ka_iv_i$, which also satisfies $\max\{a_i\}_{1\le i\le k}\le \frac1\a$. Then
$$|Tv|=\left|\sum_{i=1}^ka_iu_i\right|\le \frac k\a,$$
and
$$|(T-id)v|=\left|\sum_{i=1}^ka_i(u_i-v_i)\right|\le\frac k\a\sup_{1\le i\le k}|u_i-v_i|.$$
(\ref{E:ContOnBase1App}) follows from (\ref{E:DetLipApp}) directly.
Noting that $\det_{\eta_V,\eta_V}({\rm id})=\det ( {\rm id})=1$, and for any Lebesgue measurable $A\subset V$
$$\mu_{\eta_V}(A)=\mu(L_{\eta_V}^{-1}(A))=\mu(L_{\eta_V}^{-1}L_{\eta_W}L_{\eta_W}^{-1}(A))
=\mu((L_{\eta_V}^{-1}TL_{\eta_V}L_{\eta_W}^{-1}(A))=|\det_{\eta_V,\eta_V}(T)|\mu_{\eta_W}(A),$$
one gets (\ref{E:ContOnBase2App}) immediately.
\end{proof}

\section{Measure Theory}\label{S:MeasureTheory}
In this section, we include a basic result from measure theory. Let $X$ be a compact metric space, and denote $\mc M(X)$ the collection of probability measures on $X$. $\mc M(X)$ is a compact metrizable space endowed with the  weak-$*$ topology (measures $\mu_{i}\to\mu$ in $\mc M(X)$ if and only if $\int_{X}gd\mu_{i}\to\int_{X}gd\mu$ for all continuous functions $g:X\to \mb R$). The following lemma is elementary but we include the proof here since we need the arguments.
\begin{lemma}\label{L:ConvergenceInMeasure}
Let $\mu_{i}\to \mu $ in $\mc M(X)$. Then for any Borel $V\subset X$ with $\mu(\partial V)=0$ one has
$$\lim_{i\to\infty}\mu_{i}(V)= \mu(V).$$
\end{lemma}
\begin{proof}
First, we show that for any open set $U$ and closed set $F\subset U$
\begin{equation*}
\limsup_{i\to\infty}\mu_{i}(F)\le \mu(U).
\end{equation*}
It is trivial when $F=\emptyset$. Otherwise, by applying Urysohn's lemma,  there exists a continuous function $g:X\to[0,1]$ with support in $U$ and $g(z)=1$ when $z\in F$. Then
\begin{align*}
\limsup_{i\to\infty}\mu_{i}(F)&=\limsup_{i\to\infty}\int_{F}1d\mu_{i}=\limsup_{i\to\infty}\left(\int_{X}gd\mu_{i}-\int_{X-F}gd\mu_{i}\right)\\
&\le \limsup_{i\to\infty}\int_{X}gd\mu_{i}=\int_{X}gd\mu\le \mu(U).
\end{align*}
Note that $X-F$ is open, $X-U$ is closed and $X-U\subset X-F$. Therefore
$$1-\liminf_{i\to\infty}\mu_{i}(U)=\limsup_{i\to\infty} \mu_{i}(X-U)\le \mu(X-F)=1-\mu(F),$$
which implies that
\begin{equation*}
\liminf_{i\to\infty}\mu_{i}(U)\ge \mu(F).
\end{equation*}
Also note that $U$ is an $F_{\s}$ set. Then,  by the arbitrariness of $F$, we have that
\begin{equation}\label{E:Weak*1}
\liminf_{i\to\infty} \mu_{i}(U)\ge \mu(U).
\end{equation}
By the arbitrariness of $U$ and applying (\ref{E:Weak*1}) to $X-U$, we obtain that for any closed subset $F\subset X$
\begin{equation}\label{E:Weak*2}
\limsup_{i\to\infty}\mu_{i}(F)\le \mu(F).
\end{equation}
For any Borel set $V$, applying (\ref{E:Weak*1}) and (\ref{E:Weak*2}) on $V^{o}$, the interior of $V$, and on $\overline V$, the closure of $V$, respectively, we have that
$$\mu(V^{o})\le \liminf_{i\to\infty}\mu_{i}(V^{o})\le \limsup_{i\to\infty}\mu_{i}(\overline V)\le \mu(\overline V).$$
If $\mu(\partial V)=0$, then $\mu(V^{o})=\mu(\overline V)$ and the above inequalities become equalities. This competes the proof.
\end{proof}

 \addcontentsline{toc}{section}{References}

\bibliographystyle{plain}

 \end{document}